\documentclass[11pt,reqno]{amsart}
\usepackage[all]{xy}
\usepackage{amssymb}
\usepackage{amsthm}
\usepackage{amsmath,mathtools}
\usepackage{amscd,enumitem}
\usepackage{verbatim}
\usepackage{eurosym}
\usepackage{float}
\usepackage{color}
\usepackage{dcolumn}
\usepackage[mathscr]{eucal}
\usepackage[all]{xy}
\usepackage{bbm}
\usepackage[textheight=8.5in, textwidth=6.7in]{geometry}
\newtheorem*{conj*}{Conjecture}
\newtheorem{theorem}{Theorem}[section]

\theoremstyle{definition}
\newtheorem*{remark}{Remark}
\theoremstyle{plain}

\newtheorem{lemma}[theorem]{Lemma}

\newtheorem{prop}[theorem]{Proposition}

\newtheorem{corollary}[theorem]{Corollary}

\newcommand{\Z}{\mathbb{Z}}

\newcommand{\R}{\mathbb{R}}

\newcommand{\C}{\mathbb{C}}

\renewcommand{\pmod}[1]{\,\,({\rm mod}\,\,{#1})}
\DeclareMathOperator\Log{Log}
\numberwithin{equation}{section}

\newtheoremstyle{example}
  {\topsep}   
  {\topsep}   
  {\normalfont}  
  {0pt}       
  {\bfseries} 
  {.}         
  {5pt plus 1pt minus 1pt} 
  {}          
\theoremstyle{example}
\newtheorem*{example}{Example}

\def\({\left(}
\def\){\right)}

\def\lp{\left(}
\def\rp{\right)}

\usepackage{centernot}

\title{On the Number of Parts in Congruence Classes for Partitions into Distinct Parts}
\author{William Craig}
\address{Department of Mathematics\\University of Virginia\\ Kerchof Hall 112\\ 141 Cabell Drive \\ Charlottesville, VA 22903}
\email{wlc3vf@virginia.edu}


\begin{document}

\begin{abstract} 
For integers $0 < r \leq t$, let the function $D_{r,t}(n)$ denote the number of parts among all partitions of $n$ into distinct parts that are congruent to $r$ modulo $t$. We prove the asymptotic formula
$$D_{r,t}(n) \sim \dfrac{3^{\frac 14} e^{\pi \sqrt{\frac{n}{3}}}}{2\pi t n^{\frac 14}} \left( \log(2) + \lp \dfrac{\sqrt{3} \log(2)}{8\pi} - \dfrac{\pi}{4\sqrt{3}} \lp r - \dfrac{t}{2} \rp \rp n^{- \frac 12} \right)$$
as $n \to \infty$. A corollary of this result is that for $0 < r < s \leq t$, the inequality $D_{r,t}(n) \geq D_{s,t}(n)$ holds for all sufficiently large $n$. We make this effective, showing that for $2 \leq t \leq 10$ the inequality $D_{r,t}(n) \geq D_{s,t}(n)$ holds for all $n > 8$.
\end{abstract}

\maketitle

\section{Introduction and Statement of Results}

A {\it partition} of a positive integer $n$ is a sequence $\lambda = (\lambda_1, \lambda_2, \dots, \lambda_k)$ of weakly decreasing positive integers which sum to $n$. The integers $\lambda_i$ are called the {\it parts} of the partition $\lambda$, and we say a partition $\lambda$ has {\it distinct parts} if the $\lambda_j$ are pairwise distinct. The study of asymptotic and exact formulas for functions which count partitions of $n$ has a rich history going back to famous papers by Hardy and Ramanujan \cite{HardyRamanujan} and Rademacher \cite{RademacherExact}. 

In a different direction, Beckwith and Mertens \cite{BeckwithMertens2015, BeckwithMertens2017} have proven an asymptotic formula for the number of parts among all partitions of $n$ lying in specified congruence classes. For positive integers $1 \leq r \leq t$ and $n \geq 1$, they define the function\footnote{In \cite{BeckwithMertens2017} the function $T_{r,t}(n)$ is denoted $\widehat{T}_{r,t}(n)$.}
\begin{align*}
    T_{r,t}(n) := \sum_{\lambda \vdash n} \# \{ \lambda_j : \lambda_j \equiv r \pmod{t} \},
\end{align*}
where as usual $\lambda \vdash n$ means that $\lambda$ is a partition of $n$. In \cite{BeckwithMertens2017}, Beckwith and Mertens prove that as $n \to \infty$, we have the asymptotic formula
\begin{align*}
    T_{r,t}(n) = e^{\pi \sqrt{\frac{2n}{3}}} n^{- \frac 12} \dfrac{1}{4\pi t \sqrt{2}} \left[ \log\lp n \rp - \log\lp \frac{\pi^2}{6} \rp - 2 \lp \psi\lp \frac{r}{t} \rp + \log\lp t \rp \rp + O\lp n^{- \frac 12} \log\lp n \rp \rp \right],
\end{align*}
where $\psi(x) := \frac{\Gamma^\prime(x)}{\Gamma(x)}$ is the {\it digamma function}.

One of the main goals of this paper is to prove the analog of this result for partitions into distinct parts. In analogy with $T_{r,t}(n)$, for integers $1 \leq r \leq t$ and $n \geq 1$, we define
$$D_{r,t}(n) := \sum_{\substack{\lambda \vdash n \\ \lambda \in \mathcal{D}}} \# \{ \lambda_j : \lambda_j \equiv r \pmod{t} \},$$
where $\mathcal{D}$ denotes the collection of all partitions into distinct parts. We prove the following asymptotic formula for $D_{r,t}(n)$:

\begin{theorem} \label{MAIN}
As $n \to \infty$, we have
\begin{align*}
    D_{r,t}(n) = \dfrac{3^{\frac 14} e^{\pi \sqrt{\frac{n}{3}}}}{2\pi t n^{\frac 14}} \left( \log(2) + \lp \dfrac{\sqrt{3} \log(2)}{8\pi} - \dfrac{\pi}{4\sqrt{3}} \lp r - \dfrac{t}{2} \rp \rp n^{- \frac 12} + O\lp n^{-1} \rp \right).
\end{align*}
\end{theorem}

\begin{remark}
The asymptotic formula derived in Theorem \ref{MAIN} has an analog for similar restricted partition functions.
\end{remark}

The proof of Theorem \ref{MAIN} is in the spirit of the circle method (see for example \cite{Wright}). In many applications of the circle method, the generating functions are related to modular forms, which is critical for obtaining exact formulas. Here, the generating functions are not modular, but the theme which connects the size of the generating function near the unit disk to the size of its coefficients still applies. The growth of the generating function cannot be computed using modular transformation laws, which blocks any exact formula. However, since logarithms of $q$-Pochhammer symbols and similar types of infinite sums resemble generating functions for Bernoulli polynomials, classical Euler--Maclaurin summation provides a different route to compute asymptotic expansions. This asymptotic expansion can then be used to produce an asymptotic expansion for $D_{r,t}(n)$ via a version of the circle method originally due to Wright. \\

\begin{example}
We consider the case $t = 3$ to illustrate the accuracy of the approximation of $D_{r,3}(n)$ in Theorem \ref{MAIN}. Let $\widehat{D}_{r,t}(n)$ denote the main term of $D_{r,t}(n)$ from Theorem \ref{MAIN}. Additionally, let $Q_r(n) := \frac{D_{r,3}(n)}{ \widehat{D}_{r,3}(n)}$. The following table illustrates the convergence of $Q_r(n)$ to 1 as $n \to \infty$.

\begin{center}
\begin{tabular}{|c|c|c|c|c|c|} \hline
    $n$ & 10 & 100 & 1000 & 10000 \\ \hline
    $Q_1(n)$ & 1.159706 & 1.002613 & 1.001068 & 1.000365 \\ \hline
    $Q_2(n)$ & 0.904238 & 1.003913 & 1.001204 & 1.000378 \\ \hline
    $Q_3(n)$ & 1.167157 & 1.008440 & 1.001641 & 1.000422 \\ \hline
\end{tabular}

\vspace{0.1in}

{\small Table 1: Numerics for Theorem \ref{MAIN}.}
\end{center}
\end{example}

In their work, Beckwith and Mertens note that their asymptotic for $T_{r,t}(n)$ implies that whenever $1 \leq r < s \leq t$, the inequality $T_{r,t}(n) \geq T_{s,t}(n)$ holds for all sufficiently large $n$. We note that Theorem \ref{MAIN} gives the analogous corollary for $D_{r,t}(n)$. In fact, numerics suggest that the inequality for $T_{r,t}(n)$ has no counterexamples and that for $D_{r,t}(n)$ the counterexamples\footnote{In fact, numerics suggest that the only tuples $(r,s,n)$ which can be counterexamples are $(1,2,2)$, $(2,3,4)$, $(2,4,4)$, $(3,4,7)$, and $(4,5,8)$. Each of these are counterexamples for all sufficiently large $t$, which is clear when the relevant counts are written out explicitly.} occur only for $n \leq 8$. In support of this conjecture, we prove the following effective version of Theorem \ref{MAIN}.

\begin{theorem} \label{Effective MAIN}
For any integer $t \geq 2$ and all integers $n > \frac{400 t^2}{3}$, we have
\begin{align*}
    \left| D_{r,t}(n) - \dfrac{\log(2)}{t} V_0(n) + \dfrac{1}{2} B_1\lp \dfrac rt \rp V_1(n) - \frac{t}{8} B_2\lp \dfrac rt \rp V_2(n) + \dfrac{t^3}{192} B_4\lp \dfrac rt \rp V_4(n) \right| \leq \mathrm{Err}_t(n),
\end{align*}
where $B_n(x)$ are the Bernoulli polynomials defined in \eqref{Bernoulli Polynomial Definition}, $\mathrm{Err}_t(n)$ is defined in \eqref{E_t Definition}, and $V_s(n)$ is defined in \eqref{V Definition}.
\end{theorem}

Theorem \ref{MAIN} is proved using a variation of the circle method that traces back to Wright. This variation of the circle method uses asymptotic formulas for a given generating function to produce asymptotic formulas for its coefficients. We prove Theorem \ref{Effective MAIN} by making the asymptotic formulas fully explicit and tracing the effects of error terms throughout the circle method.

Utilizing Theorem \ref{Effective MAIN}, making the inequality $D_{r,t}(n) \geq D_{s,t}(n)$ effective requires a brief argument, combined with a finite computation which is easily carried out with the help of a computer. As a result, we obtain the following corollary.

\begin{corollary} \label{Effective Inequality}
For positive integers $1 \leq r < s \leq t$ we have $D_{r,t}(n) \geq D_{s,t}(n)$ for sufficiently large $n$. In particular, for $2 \leq t \leq 10$ this inequality holds for all $n > 8$.
\end{corollary}

The paper is organized as follows. In Section \ref{Preliminaries} we recall known results which form the framework of our approach, including Wright's circle method and Euler--Maclaurin summation. In Section \ref{Effective Euler Maclaurin} we construct the generating function of $D_{r,t}(q)$ and derive explicit bounds on infinite sums coming from Euler--Maclaurin summation. Section \ref{Applied Estimates} derives precise bounds for functions connected to the generating function for $D_{r,t}(q)$. Finally, Section \ref{Effective Proofs} then goes through the application of these tools to prove Theorems \ref{MAIN} and \ref{Effective MAIN}, as well as Corollary \ref{Effective Inequality}.

\section*{Acknowledgements}
The author thanks Ken Ono, his Ph.D advisor, and Wei-Lun Tsai for helpful discussions related to the results in this paper. The author also thanks Faye Jackson and Misheel Otgonbayar for informing the author of a mistake in a previous version of the manuscript. The author thanks the support of Ken Ono's grants, namely the Thomas Jefferson Fund and the NSF (DMS-1601306 and DMS-2055118). The author also thanks the anonymous referees for pointing out an error in the original manuscript, as well as for helpful commentary which has improved the exposition of the manuscript.

\section*{Data Availability Statement}
The author implemented a program in Sage in order to perform the finite checks at the end of the paper. This program can be obtained from the author upon reasonable request.

\section{Preliminaries} \label{Preliminaries}

\subsection{Bernoulli and Euler Polynomials}

In this section, we recall the famous {\it Bernoulli polynomials} $B_n(x)$ and {\it Euler polynomials} $E_n(x)$ and several of their properties we will need later. The generating functions for these polynomials are given in \cite[(24.2.3)]{DLMF} by
\begin{align} \label{Bernoulli Polynomial Definition}
    \sum_{n \geq 0} B_n(x) \dfrac{t^n}{n!} := \dfrac{t e^{xt}}{e^t - 1}
\end{align}
and
\begin{align*}
    \sum_{n \geq 0} E_n(x) \dfrac{t^n}{n!} := \dfrac{2 e^{xt}}{e^t + 1}.
\end{align*}
The {\it Bernoulli numbers} $B_n$ are defined by $B_n := B_n(0)$. We require a classical bound of Lehmer \cite{Lehmer} regarding the size of Bernoulli polynomials on $0 \leq x \leq 1$ (and thus also a bound for Bernoulli numbers) which says for $n \geq 2$ that
\begin{align} \label{Bernoulli Inequality}
   \left| B_n(x) \right| \leq \dfrac{2 \zeta(n) n!}{(2\pi)^{n}},
\end{align}
where $\zeta(s) := \sum_{n \geq 1} n^{-s}$ is the {\it Riemann zeta function}. We recall the fact that $B_{2n+1} = 0$ for $n > 0$ (see \cite[(24.2.2)]{DLMF}). We also require the identity
\begin{align} \label{E_n(0) Equation}
    E_n(x) = \dfrac{2}{n+1} \left[ B_{n+1}(x) - 2^{n+1} B_{n+1}\lp \frac{x}{2} \rp \right],
\end{align}
which is \cite[(24.4.22)]{DLMF}.

\subsection{Euler--Maclaurin Summation and Asymptotic Expansions}

In this section, we recall a version of the Euler--Maclaurin summation formula which has been widely used in recent years. The classical Euler--Maclaurin summation formula gives the precise difference between the integral $\int_a^b f(x) dx$ and the finite sum $f(a+1) + \dots + f(b)$ which estimates this integral. In \cite{ZagierMellin}, Zagier observed that this formula has a useful variation when $f(z)$ has a known asymptotic expansion. Here, we use asymptotic expansion in the strong sense, whereby we say $f(z) \sim \sum_{n \geq 0} b_n z^n$ provided for all $N > 0$, $f(z) - \sum_{n=0}^N b_n z^n = O(z^{N+1})$ as $z \to 0$. This asymptotic variation of the Euler--Maclaurin formula has seen increased usage in recent years, see for example \cite{BeckwithMertens2017, BCMO, BJM, BJM2}. This formula is particularly useful for computing the asymptotic growth of products of $q$-Pochhammer symbols that don't have nice modular transformation laws.

We now fix notation which will be used freely for the remainder of the paper. For $\delta > 0$, we define $D_\delta := \{ z \in \mathbb C : \left| \mathrm{arg}(z) \right| < \frac{\pi}{2} - \delta \}$. Note that if we set $z = \eta + iy$ for $\eta > 0$, then $z \in D_\delta$ if and only if $0 < |y| < M\eta$ for some constant $M > 0$ which depends on $\delta$. The {\it modified Bernoulli polynomial} $\widehat{B}_N(x)$ is the periodic function defined by $\widehat{B}_N(x) := B_N\lp x - \lfloor x \rfloor \rp$, where $\lfloor x \rfloor$ is the greatest integer less than or equal to $x$. We also use the {\it Hurwitz zeta function} $\zeta(s,x) := \sum_{n \geq 0} \frac{1}{(n + x)^{s}}$ and the {\it Euler--Mascheroni constant} $\gamma$. We furthermore set
\begin{align*}
    I_f := \int_0^\infty f(x) dx
\end{align*}
for any function $f$ for which this integral converges. The asymptotic formulas we derive require a certain decay condition of $f(x)$ at infinity, which we call {\it sufficient decay}, which holds if $f(x) = O\lp x^{-N} \rp$ as $x \to \infty$ for some $N > 1$. We may now state as a consequence of the classical Euler--Maclaurin formula the following lemma, which is a slightly rewritten form of identities appearing in \cite[Proposition 2.1]{BeckwithMertens2017}, which itself is based on the aforementioned work of Zagier.

\begin{lemma} \label{Euler-Maclaurin Exact}
Suppose that $f(z)$ is $C^\infty$ for $z$ in $D_\delta$ for some $\delta > 0$ such that $f(z)$ and all its derivatives have sufficient decay as $z \to \infty$ in $D_\delta$. Then for any real number $0 < a \leq 1$ and any positive integer $N$, we have
\begin{align*}
    \sum_{m \geq 0} f\lp (m+a)z \rp = \dfrac{1}{z} \int_{az}^\infty f(x) dx + \sum_{n=0}^{N-1} \dfrac{(-1)^n B_{n+1}}{(n+1)!} f^{(n)}(az) z^n - (-z)^N \int_0^\infty f^{(N)}\lp (x+a) z \rp \dfrac{\widehat{B}_N(x)}{N!} dx,
\end{align*}
where $f^{(N)}\lp (x+a)z \rp$ is taken to be a derivative with respect to $x$.
\end{lemma}

\begin{proof}
The proof of \cite[Proposition 2.1]{BeckwithMertens2017} implies with a slight change of variable in the last term that
\begin{align*}
    \sum_{m \geq 0} f\lp (m+a)z \rp = \dfrac{1}{z} \int_{az}^\infty f(x) dx &+ \sum_{n=0}^{N-1} \dfrac{(-1)^n B_{n+1}}{(n+1)!} f^{(n)}(az) z^n \\ &- (-1)^N \int_0^\infty \dfrac{d^N}{dx^N} \left[ f\lp (x+a)z \rp \right] \dfrac{\widehat{B}_N(x)}{N!} dx.
\end{align*}
This is equivalent to the stated formula, as evaluating the inner derivatives brings into view the factor $z^N$ in the last term.
\end{proof}

We now state the asymptotic formula of Bringmann, Jennings-Shaffer and Mahlburg, which is a generalization of \cite[Proposition 2.1]{BeckwithMertens2017} and \cite[Proposition 3]{ZagierMellin}.

\begin{prop}[{\cite[Theorem 1.2]{BJM}}] \label{Euler-Maclaurin Rapid Decay}
Suppose $0 \leq \delta < \frac{\pi}{2}$ and that $f : \C \to \C$ is holomorphic on a domain containing $D_\delta$, in particular containing the origin. Assume that $f(z)$ and all its derivatives have sufficient decay as $z \to \infty$ in $D_\delta$. Then for $a \in \R$ and $N \geq 1$ an integer, we have
$$\sum_{m \geq 0} f\lp (m + a)z \rp \sim \dfrac{I_f}{z} - \sum_{n \geq 0} c_n \dfrac{B_{n+1}(a)}{n+1} z^n$$
uniformly as $z \to 0$ in $D_\delta$.
\end{prop}

The following proposition is a refinement of Proposition \ref{Euler-Maclaurin Rapid Decay} where the function $f(z)$ is allowed to have a pole at the origin. In other words, this extends the conclusion of Proposition \ref{Euler-Maclaurin Rapid Decay} to functions $f(z)$ with principal parts $P_f(z)$ with the added property that $f(z) - P_f(z)$ has sufficient decay at infinity.

\begin{prop}[{\cite[Lemma 2.2]{BCMO}}] \label{Euler-Maclaurin Sufficient Decay}
Let $0 < a \leq 1$ and $A \in \R^+$, and assume that $f(z) \sim \sum_{n=n_0}^{\infty} c_n z^n$ $(n_0\in\Z)$ as $z \rightarrow 0$ in $D_\delta$. Furthermore, assume that $f$ and all of its derivatives are of sufficient decay in $D_\delta$. Then we have that
	\begin{align*}
	\sum_{n=0}^\infty f((n+a)z)\sim \sum_{n=n_0}^{-2} c_{n} \zeta(-n,a)z^{n}+ \frac{I_{f,A}^*}{z}-\frac{c_{-1}}{z} \left( \Log \left(Az \right) +\psi(a)+\gamma \right)-\sum_{n=0}^\infty c_n \frac{B_{n+1}(a)}{n+1} z^n,
	\end{align*}
as $z \rightarrow 0$ uniformly in $D_\delta$, where 
	\begin{align*}
		I_{f,A}^*:=\int_{0}^{\infty} \left(f(u)-\sum_{n=n_0}^{-2}c_{n}u^n-\frac{c_{-1}e^{-Au}}{u}\right)du.
	\end{align*}
\end{prop}

\subsection{Variant of Wright's Circle Method}

In this section, we recall a result of Bringmann, Ono, Males, and the author from \cite{BCMO}, which is a variation of the circle method going back to Wright \cite{Wright}. Wright's circle method gives asymptotics for the coefficients of $q$-series $F(q)$ having a nice factorization and suitable analytic properties. Given a circle $\mathcal{C}$ centered at the origin with radius less than 1, we define its {\it major arc} as that region of $\mathcal C$ where $F(q)$ is largest. In our applications, this is given by $\mathcal{C}_1 := \mathcal{C} \cap D_\delta$ for $\delta > 0$. The {\it minor arc} of $\mathcal{C}$ is then defined by $\mathcal{C}_2 := \mathcal{C} \backslash \mathcal{C}_1$. In the circle method, the integral taken over $\mathcal{C}_1$ gives the main term for the coefficients of $F(q)$ and the integral over $\mathcal{C}_2$ is merely an error term.

Here, we recall the version of Wright's circle method which we will use in the proof of Theorem \ref{MAIN}.

\begin{prop}[{\cite[Proposition 4.4]{BCMO}}] \label{Wright Circle Method}
	Suppose that $F(q)$ is analytic for $q = e^{-z}$ where $z=x+iy \in \C$ satisfies $x > 0$ and $|y| < \pi$, and suppose that $F(q)$ has an expansion $F(q) = \sum_{n=0}^\infty c(n) q^n$ near 1. Let $N,M>0$ be fixed constants. Consider the following hypotheses:
	
	\begin{enumerate}
		\item[\rm(1)] As $z\to 0$ in the bounded cone $|y|\le Mx$ (major arc), we have
		\begin{align*}
			F(e^{-z}) = C z^{B} e^{\frac{A}{z}} \left( \sum_{j=0}^{N-1} \alpha_j z^j + O_\delta\left(|z|^N\right) \right),
		\end{align*}
		where $\alpha_s \in \C$, $A,C \in \R^+$, and $B \in \R$. 
		
		\item[\rm(2)] As $z\to0$ in the bounded cone $Mx\le|y| < \pi$ (minor arc), we have 
		\begin{align*}
			\lvert	F(e^{-z}) \rvert \ll_\delta e^{\frac{1}{\mathrm{Re}(z)}(A - \kappa)},
		\end{align*}
		for some $\kappa\in \R^+$.
	\end{enumerate}
	If  {\rm(1)} and {\rm(2)} hold, then as $n \to \infty$ we have for any $N\in \R^+$ 
	\begin{align*}
		c(n) = C n^{\frac{1}{4}(- 2B -3)}e^{2\sqrt{An}} \lp \sum\limits_{r=0}^{N-1} p_r n^{-\frac{r}{2}} + O\left(n^{-\frac N2}\right) \rp,
	\end{align*}
	where $p_r := \sum\limits_{j=0}^r \alpha_j c_{j,r-j}$ and $c_{j,r} := \dfrac{(-\frac{1}{4\sqrt{A}})^r \sqrt{A}^{j + B + \frac 12}}{2\sqrt{\pi}} \dfrac{\Gamma(j + B + \frac 32 + r)}{r! \Gamma(j + B + \frac 32 - r)}$. 
\end{prop}

\begin{remark}
The constant $C$ in Proposition \ref{Wright Circle Method} does not appear in \cite{BCMO}, but is trivially equivalent to the result in \cite{BCMO} by factoring out $C$ from each $\alpha_i$.
\end{remark}

\subsection{Estimates with Bessel Functions}

We now consider certain estimates with Bessel functions which we will require when effectively implementing Wright's circle method. Recall that the {\it modified Bessel function} $I_\nu(z)$ is defined for any $\nu \in \C$ by
\begin{align*}
    I_\nu(x) := \lp \dfrac x2 \rp^\nu  \dfrac{1}{2\pi i} \int_{\mathcal D} t^{-\nu-1} \exp\lp \dfrac{x^2}{4t} + t \rp dt,
\end{align*}
where $\mathcal D$ is any contour running from $-\infty$ below the negative real axis, counterclockwise around 0, and back to $-\infty$ above the negative real axis. We shall choose $\mathcal D = \mathcal D_- \cup \mathcal D_0 \cup \mathcal D_+$, each of which depend on a particular choice of $z = \eta + iy$ with $\eta = \frac{\pi}{\sqrt{12n}}$ for $n > 0$. These components of $\mathcal D$ are given by
\begin{align*}
    \mathcal D_{\pm} := \{ u + iv \in \C : u \leq \eta, v = \pm 10 \eta \}, \\
    \mathcal D_0 := \{ u + iv \in \C : u = \eta, |v| \leq 10 \eta \}.
\end{align*}
Note that this dependence on $z$ does not change the value of the integral, since one can shift the paths of integration. We shall compare the size of $I_\nu(z)$ to its main term. In particular, define
\begin{align*}
    \widehat{I}_\nu(n) := \lp \dfrac{\pi^2}{12n} \rp^{\frac{\nu}{2}} \dfrac{1}{2\pi i} \int_{\mathcal D_0} t^{-\nu-1} \exp\lp \dfrac{\pi^2}{12t} + \lp n + \dfrac{1}{24} \rp t \rp dt.
\end{align*}
The following lemma shows how $\widehat{I}_\nu(n)$ approximates $I_\nu(z)$ for certain values of $z$.

\begin{lemma} \label{Bessel Estimates}
Let $n \geq 1$ be an integer and $\nu \leq -1$. Then
\begin{align*}
	\left| I_\nu\lp \pi \sqrt{\dfrac{1}{3}\lp n + \dfrac{1}{24} \rp} \rp - \widehat{I}_\nu(n) \right| < 2 \lp \dfrac{2\pi^2}{24n+1} \rp^{\frac{\nu}{2}} \exp\lp \dfrac{3\pi}{4} \sqrt{\dfrac{n}{3}} \rp \int_0^\infty \lp 10 + u \rp^{-\nu - 1} e^{-\lp n + \frac{1}{24} \rp u} du.
\end{align*}
\end{lemma}

\begin{proof}
By a change of variables $t \mapsto \lp n + \frac{1}{24} \rp t$ and shifting of the path of integration back to $\mathcal D$, we see that
\begin{align*}
    I_\nu\lp \pi \sqrt{\dfrac{1}{3}\lp n + \dfrac{1}{24} \rp} \rp = \lp \dfrac{\pi^2}{12\lp n + \frac{1}{24} \rp} \rp^{\frac{\nu}{2}} \dfrac{1}{2\pi i} \int_{\mathcal D} t^{-\nu-1} \exp\lp \dfrac{\pi^2}{12t} + \lp n + \dfrac{1}{24} \rp t \rp dt.
\end{align*}
Thus, we have
\begin{align*}
    I_\nu\lp \pi \sqrt{\dfrac{1}{3}\lp n + \dfrac{1}{24} \rp} \rp - \widehat{I}_\nu(n) = \lp \dfrac{2\pi^2}{24n+1} \rp^{\frac{\nu}{2}} \dfrac{1}{2\pi i} \int_{\mathcal D_+ \cup \mathcal D_-} t^{-\nu-1} \exp\lp \dfrac{\pi^2}{12t} + \lp n + \dfrac{1}{24} \rp t \rp dt.
\end{align*}
For $t \in \mathcal D_-$, we may set $t = \lp \eta - u \rp - 10 \eta i$. Since we have $\mathrm{Re}\lp \frac{\pi^2}{12t} \rp \leq \frac{\pi}{4} \sqrt{\frac{n}{3}}$ for all $u \geq 0$ and $|t| \leq |\eta - u| + |10 \eta i| < 11\eta + u = \frac{11\pi}{\sqrt{12n}} + u$, we have
\begin{align*}
    \left| t^{-\nu - 1} \exp\lp \dfrac{\pi^2}{12t} + nt \rp \right| &\leq |t|^{-\nu - 1} \exp\lp \dfrac{\pi}{4} \sqrt{\dfrac{n}{3}} + \lp n + \dfrac{1}{24} \rp \lp \eta - u \rp \rp \\ &\leq \lp \dfrac{11 \pi}{\sqrt{12n}} + u \rp^{-\nu - 1} \exp\lp \dfrac{3\pi}{4} \sqrt{\dfrac{n}{3}} - \lp n + \dfrac{1}{24} \rp u \rp,
\end{align*}
where the last inequality uses $-\nu - 1 \geq 0$. The same bound holds for $\mathcal D_+$. Since $\frac{11 \pi}{\sqrt{12n}} < 10$, we conclude that
\begin{align*}
    \left| I_\nu\lp \pi \sqrt{\dfrac{1}{3}\lp n + \dfrac{1}{24} \rp} \rp - \widehat{I}_\nu(n) \right| < 2 \lp \dfrac{2\pi^2}{24n+1} \rp^{\frac{\nu}{2}} \exp\lp \dfrac{3\pi}{4} \sqrt{\dfrac{n}{3}} \rp \int_0^\infty \lp 10 + u \rp^{-\nu - 1} e^{-\lp n + \frac{1}{24} \rp u} du.
\end{align*}
This completes the proof.
\end{proof}

\section{Generating Functions and Effective Euler--Maclaurin Asymptotics} \label{Effective Euler Maclaurin}

The first part of this section derives the generating function $\mathcal{D}_{r,t}(q)$ of $D_{r,t}(n)$. We then prove that this generating function has a direct connection to the Euler--Maclaurin framework. Following the discussion of $\mathcal D_{r,t}(q)$, we show how to make the error terms in Propositions \ref{Euler-Maclaurin Rapid Decay} and \ref{Euler-Maclaurin Sufficient Decay} effective.

\subsection{Generating Function for $D_{r,t}(n)$} \label{Generating Function, L, and Xi}

This section is dedicated to defining the generating function for $D_{r,t}(n)$ and an important factorization of this generating function. Define
\begin{align*}
    \mathcal{D}_{r,t}(q) := \sum_{n \geq 0} D_{r,t}(n) q^n.
\end{align*}
We also use the standard {\it $q$-Pochhammer symbol} $(a;q)_\infty$, which is defined by $$(a;q)_\infty := \prod_{n \geq 1} \lp 1 - a q^{n-1} \rp$$ for $|q| < 1$. Recall that $(-q;q)_\infty$ is the generating function for the number of partitions of $n$ into distinct parts, as each term $(1 + q^m)$ appearing in the product dictates whether a given partition has a part of size $m$. By a slight modification of this argument, we obtain $\mathcal{D}_{r,t}(q)$.

\begin{lemma} \label{Generating Function}
We have the generating function identity
$$\mathcal{D}_{r,t}(q) = (-q;q)_\infty \sum_{k \geq 0} \dfrac{q^{kt + r}}{1 + q^{kt + r}}.$$
\end{lemma}

\begin{proof}
By modifying Euler's generating function $\lp -q;q \rp_\infty$ for partitions into distinct parts, we see that $\frac{q^m}{1 + q^m} (-q;q)_\infty$ is the generating function for partitions into distinct parts which include $m$ as a part. Furthermore, since all parts are distinct, this is also the generating function for $D_{r,t}(n)$. Therefore, summing over $m$ equivalent to $r$ modulo $t$ yields
$$\mathcal{D}_{r,t}(q) = \sum_{\substack{m \geq 0 \\ m \equiv r \pmod{t}}} \dfrac{q^m (-q;q)_\infty}{1 + q^m} = (-q;q)_\infty \sum_{k \geq 0} \dfrac{q^{kt + r}}{1 + q^{kt + r}}.$$
This completes the proof.
\end{proof}

We conclude this section with a brief lemma regarding a natural decomposition of this generating function, which will be useful for computing asymptotics. Define the functions $\xi(q) := (-q;q)_\infty$ and $L_{r,t}(q) := \sum_{k \geq 0} \frac{q^{kt+r}}{1 + q^{kt+r}}$, so that $\mathcal{D}_{r,t}(q) = \xi(q) L_{r,t}(q)$. Additionally, define $B(z) := \frac{e^{-z}}{z\lp 1 - e^{-z} \rp}$ and $E(z) := \frac{e^{-z}}{1 + e^{-z}}$. This notation is assumed throughout the remainder of the paper. The importance of the functions $B(z)$ and $E(z)$ comes from the following series expansions connecting them to $\mathcal{D}_{r,t}(q)$, which we record now for convenience. Throughout the remainder of the paper, we let $\Log(z)$ denote the principal branch of the logarithm.

\begin{lemma} \label{Xi and L z-Expansions}
Let $\xi(q)$ and $L_{r,t}(q)$ be defined as above. Then, for $q = e^{-z}$ with $\mathrm{Re}(z) > 0$, we have
$$\Log\xi\lp e^{-z} \rp = z \lp \sum_{m \geq 0} B\lp \lp m + \frac 12 \rp 2z \rp - \sum_{m \geq 0} B\lp \lp m+1 \rp 2z \rp \rp$$
and
$$L_{r,t}\lp e^{-z} \rp = \sum_{k \geq 0} E\lp \lp k + \frac{r}{t} \rp tz \rp.$$
\end{lemma}

\begin{proof}
Expanding $\Log \xi(q)$ as a Taylor series, we have
\begin{align*}
    \Log \xi(q) = \sum_{n \geq 1} \Log\lp 1 + q^n \rp = - z \sum_{m \geq 1} \dfrac{(-1)^m q^m}{mz \lp 1 - q^m \rp}.
\end{align*}

For $q = e^{-z}$, it follows from the definition of $B(z)$ that
\begin{align*}
    \Log \xi\lp e^{-z} \rp = z \lp \sum_{m \geq 0} B\lp \lp m + \frac 12 \rp 2z \rp - \sum_{m \geq 0} B\lp \lp m+1 \rp 2z \rp \rp.
\end{align*}
This proves the first part of the lemma. The second is an analogous calculation with $E(z)$ in place of $B(z)$, i.e.
\begin{align*}
    L_{r,t}\lp e^{-z} \rp = \sum_{k \geq 0} \dfrac{e^{-(kt + r)z}}{1 + e^{-(kt+r)z}} = \sum_{k \geq 0} E\lp \lp k + \frac{r}{t} \rp tz \rp.
\end{align*}
This completes the proof.
\end{proof}

We also record the Taylor expansions of $B(z)$ and $E(z)$ for later use. From the fact that $\frac{z}{e^z \pm 1} = \frac{z e^{-z}}{1 \pm e^{-z}}$ the generating function for the Bernoulli numbers $B_n$ is given by $B(z) = \frac{1}{z^2} - \frac{1}{2z} + \sum_{n \geq 0} \frac{B_{n+2}}{(n+2)!} z^n$, and similarly $E(z) = \sum_{n \geq 0} \frac{e_n}{n!} z^n$, where $e_n := \frac{E_n(0)}{2}$. We note for later that by \eqref{E_n(0) Equation},
\begin{align} \label{e_n Evaluation}
    e_n = \dfrac{1 - 2^{n+1}}{n+1} B_{n+1}.
\end{align}

\subsection{Effective Estimates}

In the proof of Theorem \ref{Effective MAIN}, we require the error terms in Proposition \ref{Euler-Maclaurin Rapid Decay} to be made effective. This is achieved by simply keeping track of the higher degree terms that were dropped in the proof of Proposition \ref{Euler-Maclaurin Rapid Decay}. In our applications, we will only require explicit bounds for the functions $E(z)$ and $B(z)$, functions which satisfy the conditions of Propositions \ref{Euler-Maclaurin Rapid Decay} and \ref{Euler-Maclaurin Sufficient Decay}, respectively. These two propositions essentially follow from ``erasing" higher-order terms in Lemma \ref{Euler-Maclaurin Exact}. Therefore, making the error terms in these results effective is essentially a matter of bookkeeping. These effective error terms become the central tool for implementing an effective version of Wright's circle method.

\begin{prop} \label{Euler-Maclaurin Rapid Decay Effective}
Let $f(z)$ be $C^\infty$ in $D_\delta$ with power series expansion $f(z) = \sum_{n \geq 0} c_n z^n$ that converges absolutely in the region $0 \leq |z| < R$ for some positive constant $R$, and let $f(z)$ and all its derivatives have sufficient decay as $z \to \infty$ in $D_\delta$. Then for any real number $0 < a \leq 1$ and any integer $N > 0$,
\begin{align*}
    \left| \sum_{m \geq 0} f\lp (m+a)z \rp - \dfrac{I_f}{z} + \sum_{n = 0}^{N-1} c_n \dfrac{B_{n+1}(a)}{n+1} z^n \right| \leq \dfrac{M_{N+1} J_{f,N+1}(z)}{(N+1)!} |z|^N + \sum_{k \geq N} |c_k| \lp 1 + \dfrac{k!}{10 (k-N)!} \rp |z|^k,
\end{align*}
where $M_{N+1} := \max\limits_{0 \leq x \leq 1} \left| B_{N+1}(x) \right|$ and
\begin{align*}
    J_{f,N+1}(z) := \int_0^\infty \left| f^{(N+1)}\lp w \rp \right| |dw|,
\end{align*}
where the path of integration proceeds along the line through the origin and $z$.
\end{prop}

\begin{proof}
From Proposition \ref{Euler-Maclaurin Rapid Decay}, we already know that
\begin{align*}
    S_N(z) := \sum_{m \geq 0} f\lp (m+a)z \rp - \dfrac{I_f}{z} + \sum_{n=0}^{N-1} c_n \dfrac{B_{n+1}(a)}{n+1} z^n = O_N\lp z^N \rp.
\end{align*}
It suffices to make this upper bound effective. We use the shorthand $$J_{N+1,a}(z) := \int_{az}^\infty f^{(N+1)}\lp w \rp \dfrac{\widehat{B}_{N+1}\lp \frac wz - a \rp}{(N+1)!} dw,$$
which is the integral from last term of Lemma \ref{Euler-Maclaurin Exact} with a substitution $w = \lp x+a \rp z$. By Lemma \ref{Euler-Maclaurin Exact}, we may write
\begin{align*}
    S_N(z) = \dfrac{-1}{z} \int_0^{az} f(x) dx + \sum_{n=0}^{N} \dfrac{(-1)^n B_{n+1}}{(n+1)!} f^{(n)}(az) z^n + \sum_{n=0}^{N} c_n \dfrac{B_{n+1}(a)}{n+1} z^n - (-z)^N J_{N+1,a}(z).
\end{align*}
Because $0 < a \leq 1$ and $0 < |z| < R$, we have $|az| < R$ and so we may expand $f(x)$ and its derivatives as power series for $0 \leq x \leq |az|$. Using these power series representations and the absolute convergence of $\int_0^{az} f(x) dx$, we have
\begin{align*}
    S_N(z) = - \sum_{k \geq 0} \dfrac{c_k}{k+1} a^{k+1} z^k &+ \sum_{n=0}^{N} \dfrac{(-1)^n B_{n+1}}{(n+1)!} \sum_{k \geq 0} \dfrac{(k+n)!}{k!} c_{n+k} a^k z^{n+k} \\ &+ \sum_{n=0}^{N} c_n \dfrac{B_{n+1}(a)}{n+1} z^n - (-z)^N J_{N+1,a}(z).
\end{align*}
It is already known, for instance by Proposition \ref{Euler-Maclaurin Rapid Decay}, that $S_N(z) = O(z^N)$, so the lower-order terms in the above identity necessarily cancel. Thus, we have
\begin{align*}
    S_N(z) &= - \sum_{k \geq N} \dfrac{c_k}{k+1} a^{k+1} z^k + \sum_{n=0}^{N} \dfrac{(-1)^n B_{n+1}}{(n+1)!} \sum_{k \geq N-n} c_{n+k} \dfrac{(n+k)!}{k!} a^k z^{n+k} - (-z)^N J_{N+1,a}(z).
\end{align*}
By taking $k \mapsto k-n$ in the second term and rearranging, we obtain
\begin{align*}
    S_N(z) &= \sum_{k \geq N} c_k \left[- \dfrac{a^{k+1}}{k+1}  + \sum_{n=0}^N \dfrac{1}{n+1}\left[ (-1)^n B_{n+1} \binom{k}{n} a^{k-n} \right] \right] z^k - \lp -z \rp^N J_{N+1,a}(z).
\end{align*}

We now bound the remaining terms. The integral $J_{N+1,a}(z)$ is bounded trivially by
\begin{align*}
    \left| J_{N+1,a}(z) \right| \leq \dfrac{M_{N+1}}{(N+1)!} J_{f,N+1}(z) = O_N(1)
\end{align*}
since $f(z)$ is bounded near zero and has sufficient decay as $z \to \infty$ in $D_\delta$.

We also have, using Lehmer's bound \eqref{Bernoulli Inequality} and elementary estimates that for $k \geq N$,
\begin{align*}
    \left| - \dfrac{a^{k+1}}{k+1} + \sum_{n=0}^N \dfrac{1}{n+1} \left[ (-1)^n B_{n+1} \binom{k}{n} a^{k-n} \right] \right| &\leq \dfrac{a^{k+1}}{k+1} + \dfrac{a^k}{2} + \sum_{\substack{n=1 \\ n \text{ odd}}}^N \dfrac{2 \zeta(n+1) n!}{(2\pi)^{n+1}} \binom{k}{n} a^{k-n} \\ &< \dfrac{1}{k+1} + \dfrac{1}{2} + \dfrac{\pi}{6} \sum_{\substack{n=1 \\ n \text{ odd}}}^N \dfrac{k!}{(2 \pi)^n (k-n)!}.
\end{align*}
Since $1 \leq n \leq N \leq k$, $\frac{k!}{(k-n)!} < \frac{k!}{(k-N)!}$, and $\frac{\pi}{6} \sum_{n \geq 0} \frac{1}{(2\pi)^{2n+1}} < \frac{1}{10}$,
\begin{align*}
    \left| - \dfrac{a^{k+1}}{k+1} + \sum_{n=0}^N \dfrac{1}{n+1} \left[ (-1)^n B_{n+1} \binom{k}{n} a^{k-n} \right] \right| < 1 + \dfrac{k!}{10 (k-N)!}.
\end{align*}
Thus,
\begin{align*}
    \left| \sum_{k \geq N} c_k \left[ - \dfrac{a^{k+1}}{k+1} + \sum_{n=0}^N \dfrac{1}{n+1} \left[ (-1)^n B_{n+1} \binom{k}{n} a^{k-n} \right] \right] z^k \right| \leq \sum_{k \geq N} |c_k| \lp 1 + \dfrac{k!}{10 (k-N)!} \rp |z|^k.
\end{align*}
Combining all bounds completes the proof.
\end{proof}

The proposition above shows how Euler--Maclaurin summation can be used to derive effective asymptotics for certain infinite series involving a function $f(z)$ with rapid decay at infinity. In analogy with Proposition \ref{Euler-Maclaurin Sufficient Decay}, we now show how to derive explicit bounds for the case of sufficient decay at infinity.

\begin{prop} \label{Euler-Maclaurin General Effective}
Let $f(z)$ be $C^\infty$ in $D_\delta$ with Laurent series $f(z) = \sum_{n = n_0}^\infty c_n z^n$ that converges absolutely in the region $0 < |z| < R$ for some positive constant $R$. Suppose $f(z)$ and all its derivatives have sufficient decay as $z \to \infty$ in $D_\delta$. Then for any real numbers $0 < a \leq 1$, $A > 0$ and any integer $N > 0$, we have
\begin{align*}
    \bigg| \sum_{m \geq 0} f\lp (m+a)z \rp - \sum_{n = n_0}^{-2} c_n \zeta(-n,a) z^n &- \dfrac{I_{f,A}^*}{z} + \dfrac{c_{-1}}{z} \lp \Log\lp Az \rp + \gamma + \psi\lp a \rp \rp + \sum_{n \geq 0} c_n^* \dfrac{B_{n+1}(a)}{n+1} z^n \bigg| \\ &\leq \dfrac{M_{N+1} J_{g,N+1}(az)}{(N+1)!} |z|^N + \sum_{k \geq N} |b_k| \lp 1 + \dfrac{k!}{10 (k-N)!} \rp |z|^k,
\end{align*}
where $g(z) := f(z) - \frac{c_{-1} e^{-Az}}{z} - \sum_{n = n_0}^{-2} c_n z^n$, $b_n := c_n - \frac{(-A)^{n+1} c_{-1}}{(n+1)!}$, $M_N$ and $J_{g,N}$ are defined as in Lemma \ref{Euler-Maclaurin Rapid Decay Effective}, and
\begin{align*}
    c_n^* := \begin{cases} c_n & \text{if } n \leq N-1, \\ \dfrac{(-A)^{n+1} c_{-1}}{(n+1)!} & \text{if } n \geq N. \end{cases}
\end{align*}
\end{prop}

\begin{proof}
Since
$$f(z) = g(z) + \dfrac{c_{-1} e^{-Az}}{z} + \sum_{n = n_0}^{-2} c_n z^n,$$
then $g(z)$ is holomorphic at $z=0$ and has sufficient decay at infinity. Because $f(z)$ has a Laurent series converging for $0 < |z| < R$, it follows that $g(z)$ has a Taylor series $g(z) = \sum_{n = 0}^\infty b_n z^n$ which converges for $|z| < R$. Also note that $I_g = I^*_{f,A}$ by definition. Therefore, Proposition \ref{Euler-Maclaurin Rapid Decay Effective} implies for $N > 0$ that
\begin{align*}
    \left| \sum_{m \geq 0} g\lp (m+a)z \rp - \dfrac{I_{f,A}^*}{z} + \sum_{n = 0}^{N-1} b_n \dfrac{B_{n+1}(a)}{n+1} z^n \right| \leq \dfrac{M_{N+1} J_{g,N+1}(z)}{(N+1)!} |z|^N + \sum_{k \geq N} |b_k| \lp 1 + \dfrac{k!}{10 (k-N)!} \rp |z|^k
\end{align*}
for $z \in D_\delta$ with $0 < |z| < R$. From the definition of $g(z)$ this becomes
\begin{align*}
    \Bigg| \sum_{m \geq 0} \left[ f\lp (m+a)z \rp - \dfrac{c_{-1} e^{-A(m+a)z}}{(m+a)z} \right] &- \sum_{n = n_0}^{-2} c_n \zeta(-n,a) z^n - \dfrac{I_{f,A}^*}{z} + \sum_{n = 0}^{N-1} b_n \dfrac{B_{n+1}(a)}{n+1} z^n \Bigg| \\ &\leq\dfrac{M_{N+1} J_{g,N+1}(z)}{(N+1)!} |z|^N + \sum_{k \geq N} |b_k| \lp 1 + \dfrac{k!}{10 (k-N)!} \rp |z|^k.
\end{align*}
By the definition of $b_n$ we have
\begin{align*}
    \sum_{n = 0}^{N-1} b_n \dfrac{B_{n+1}(a)}{n+1} z^n = \sum_{n = 0}^{N-1} c_n \dfrac{B_{n+1}(a)}{n+1} z^n - \sum_{n = 0}^{N-1} \dfrac{(-A)^{n+1} c_{-1}}{(n+1)!} \dfrac{B_{n+1}(a)}{n+1} z^n,
\end{align*}
and if we adopt the notation
\begin{align*}
    \dfrac{c_{-1}}{z} H_{a,N}(z) := \dfrac{c_{-1}}{z} \lp \sum_{m \geq 0} \dfrac{e^{-A(m+a)z}}{m+a} + \sum_{n=0}^{N-1} \dfrac{B_{n+1}(a)}{(n+1) (n+1)!} (-Az)^{n+1} \rp,
\end{align*}
it follows that
\begin{align*}
    \bigg| \sum_{m \geq 0} f\lp (m+a)z \rp - \sum_{n = n_0}^{-2} c_n \zeta(-n,a) z^n -& \dfrac{c_{-1}}{z} H_{a,N}(Az) - \dfrac{I_{f,A}^*}{z} + \sum_{n = 0}^{N-1} c_n \dfrac{B_{n+1}(a)}{n+1} z^n \bigg| \\ &\leq \dfrac{M_{N+1} J_{g,N+1}(z)}{(N+1)!} |z|^N + \sum_{k \geq N} |b_k| \lp 1 + \dfrac{k!}{10 (k-N)!} \rp |z|^k.
\end{align*}
By \cite[Equation 5.10]{BJM}, it is known that
\begin{align*}
    H_a(z) := \sum_{n \geq 0} \dfrac{e^{-(m+a)z}}{m+a} + \sum_{n \geq 0} \dfrac{B_{n+1}(a)}{(n+1) (n+1)!} (-z)^{n+1}
\end{align*}
satisfies $H_a(Az) = - \Log(Az) - \gamma - \psi(a)$ for any $A > 0$. Since $$H_{a,N}(Az) = H_a(Az) - \sum_{n \geq N} \frac{B_{n+1}(a)}{(n+1) (n+1)!} (-Az)^{n+1},$$
this completes the proof.
\end{proof}

\section{Estimates for $L_{r,t}(q)$ and $\xi(q)$} \label{Applied Estimates}

In this section, we prove effective bounds for the functions $L_{r,t}(q)$ and $\xi(q)$ on both the major and minor arcs. The first subsection covers major arc bounds, and the second covers minor arc bounds.

\subsection{Major arc effective bounds}

In this subsection, we compute effective bounds on the functions $L_{r,t}(q)$ and $\xi(q)$ on the major arc. We also note that in the region $0 \leq |y| < 10\eta$, the hypothesis $\eta < \frac{\pi}{40t}$ always implies $|z| < \frac{\sqrt{101} \pi}{80} < \frac{2}{5}$.

\begin{lemma} \label{L Major Arc}
Let $t \geq 2$ and $0 < r \leq t$ be integers and $z = \eta + iy$ a complex number satisfying $0 \leq |y| < 10\eta$ and $\eta < \frac{\pi}{40t}$. Then
\begin{align*}
     \left| L_{r,t}\lp e^{-z} \rp - \dfrac{\log(2)}{tz} + \dfrac{1}{2} B_1\lp \dfrac{r}{t} \rp - \dfrac{t}{8} B_2\lp\frac{r}{t}\rp z + \dfrac{t^3}{192} B_4\lp \frac rt \rp z^3 \right| < \dfrac{1}{20} t^5 |z|^5.
\end{align*}
\end{lemma}

\begin{proof}
The proof relies on an application of Proposition \ref{Euler-Maclaurin Rapid Decay Effective} to $E(z) = \sum_{n=0}^\infty \frac{e_n}{n!} z^n$, whose radius of convergence is $\pi$. We note $M_6 = \frac{1}{42}$. Thus, applying Proposition \ref{Euler-Maclaurin Rapid Decay Effective} to $E(z) = \sum_{k \geq 0} \frac{e_k}{k!} z^k$ with $a = \frac rt$, we obtain
\begin{align*}
	\bigg| \sum_{k \geq 0} E\lp \lp k + \frac{r}{t} \rp z \rp - \dfrac{I_E}{z} + \dfrac{1}{2} B_1\lp \dfrac rt \rp - \dfrac{1}{8} B_2\lp \dfrac rt \rp z &+ \dfrac{1}{192} B_4\lp \dfrac rt \rp z^3 \bigg| \\ &\leq \dfrac{J_{E,6}(z)}{30240} |z|^5 + \sum_{k \geq 5} |e_k| \lp 1 + \dfrac{k!}{10 (k-5)!} \rp |z|^k.
\end{align*}
We also have $I_E = \int_0^\infty \frac{dx}{e^x + 1} = \log(2)$ and therefore by Lemma \ref{Xi and L z-Expansions} we have
\begin{align*}
	\bigg| L_{r,t}\lp e^{-z} \rp - \dfrac{\log(2)}{tz} + \dfrac{1}{2} B_1\lp \dfrac{r}{t} \rp - \dfrac{t}{8} B_2\lp\frac{r}{t}\rp z &+ \dfrac{t^3}{192} B_4\lp \frac rt \rp z^3 \bigg| \\ &\leq \dfrac{J_{E,6}(z)}{30240} |tz|^5 + |tz|^5 \sum_{k \geq 5} |e_k| \lp 1 + \dfrac{k!}{10 (k-5)!} \rp |tz|^{k-5},
\end{align*}
which is valid for all for all $|z| < \frac{\pi}{t}$, hence in particular when $\eta < \frac{\pi}{40t}$ and $0 \leq |y| < 10\eta$. We now proceed to estimate each piece on the right-hand side.

Now, let $\alpha = \frac{\pi}{2} \frac{z}{|z|}$. Then we bound $J_{E,6}(z)$ by the decomposition
\begin{align*}
	J_{E,6}(z) = \int_0^\alpha \left| E^{(6)}(w) \right|dw + \int_\alpha^\infty \left| E^{(6)}(w) \right|dw.
\end{align*}
The function $E^{(6)}(z)$ is given by
\begin{align*}
	E^{(6)}(z) = \dfrac{e^z\lp e^z - 1 \rp \lp e^{4z} - 56 e^{3z} + 246 e^{2z} - 56 e^z + 1 \rp}{\lp e^z + 1 \rp^7}.
\end{align*}
By the triangle inequality, we have
\begin{align*}
	\left| E^{(6)}(z) \right| \leq \dfrac{e^\eta \lp e^\eta + 1 \rp \lp e^{4\eta} + 56 e^{3\eta} + 246 e^{2\eta} + 56 e^\eta + 1 \rp}{\lp e^\eta - 1 \rp^7}.
\end{align*}
These bounds entail that for $u = \mathrm{Re}\lp w \rp$ and the major arc $0 \leq |\mathrm{Im}(w)| < 10 u$, we have
\begin{align*}
	\int_{\alpha}^\infty \left| E^{(6)}(w) \right| dw \leq \sqrt{101} \int_{\pi/2}^\infty \dfrac{e^u \lp e^u + 1 \rp \lp e^{4u} + 56 e^{3u} + 246 e^{2u} + 56 e^u + 1 \rp}{\lp e^u - 1 \rp^7} |du| < 81.
\end{align*}
The power series representation of $E^{(6)}(w)$ is valid in the region from $0$ to $\alpha$. Combining the estimates $|w| < \frac{\pi}{2}$, \eqref{Bernoulli Inequality}, \eqref{e_n Evaluation}, the vanishing of $B_{2n+1}$ for $n \geq 1$, and the fact that $\zeta(n)$ is decreasing for $n > 1$, we have
\begin{align*}
	\left| E^{(6)}(w) \right| \leq \sum_{k=6}^\infty \dfrac{2^{k+1} \left|B_{k+1}\right| \pi^{k-6}}{(k-6)! 2^{k-6}} \leq \dfrac{\zeta(8) 2^7}{\pi^7} \sum_{k \geq 3} \dfrac{(2k+2)!}{2^{2k+1} (2k-5)!} < 429.
\end{align*}
Therefore,
\begin{align*}
	J_{E,6}(z) < \dfrac{429\pi}{2} + 81 < 755.
\end{align*}
We may also show using \eqref{Bernoulli Inequality} and \eqref{e_n Evaluation} that $\left| \frac{e_n}{n!} \right| \leq \frac{\pi}{3} \cdot \lp \frac{1}{\pi} \rp^n$, and therefore since $|z| < \frac{\sqrt{101}}{40t}$ we have
\begin{align*}
	\sum_{k \geq 5} |e_k| \lp 1 + \dfrac{k!}{10 (k-5)!} \rp |tz|^{k-5} < \dfrac{1}{3\pi^4} \sum_{k \geq 5} \lp 1 + \dfrac{k!}{10 (k-5)!} \rp \lp \dfrac{\sqrt{101}}{40} \rp^{k-5} < \dfrac{1}{4}.
\end{align*}
Thus,
\begin{align*}
	\left| L_{r,t}\lp e^{-z} \rp - \dfrac{\log(2)}{tz} + \dfrac{1}{2} B_1\lp \dfrac{r}{t} \rp - \dfrac{t}{8} B_2\lp\frac{r}{t}\rp z + \dfrac{t^3}{192} B_4\lp \frac rt \rp z^3 \right| \leq \dfrac{755}{30240} |tz|^5 + \dfrac{|tz|^5}{4} < \dfrac{7}{25} t^5 |z|^5.
\end{align*}
This completes the proof.
\end{proof}

\begin{corollary} \label{L Major Arc Corollary}
Let $0 < r \leq t$ be integers and $z = \eta + iy$ a complex number satisfying $0 \leq |y| < 10\eta$ and $\eta < \frac{\pi}{40t}$. Then
\begin{align*}
    \left| L_{r,t}\lp e^{-z} \rp \right| < \dfrac{14}{|tz|}.
\end{align*}
\end{corollary}

\begin{proof}
By the triangle inequality and Lemma \ref{L Major Arc}, we have
\begin{align*}
    \left| L_{r,t}\lp e^{-z} \rp \right| < \dfrac{\log(2)}{t|z|} + \left| \dfrac 12 B_1\lp \dfrac rt \rp \right| + \left| \dfrac t8 B_2\lp \dfrac rt \rp z \right| + \left| \dfrac{t^3}{192} B_4\lp \dfrac rt \rp z^3 \right| + \dfrac{7}{25} |tz|^5.
\end{align*}
The fact that $\eta < \frac{\pi}{40t}$ entails $|z| < \frac{\sqrt{101} \pi}{40t} < \frac{4}{5t}$. Using the trivial bound on $B_1\lp \frac rt \rp$, Lehmer's bound \eqref{Bernoulli Inequality} and $|tz| < \frac{4}{5}$, we obtain
\begin{align*}
    \left| L_{r,t}\lp e^{-z} \rp \right| < \dfrac{\log(2) + \frac{1}{4}|tz| + \frac{5}{96}|tz|^2 + \frac{1}{1344} |tz|^4 + \frac{7}{25} |tz|^6}{|tz|} < \dfrac{14}{|tz|},
\end{align*}
which completes the proof.
\end{proof}

\begin{lemma} \label{Xi Major Arc}
For any integer $t \geq 2$ and any complex number $z = \eta + iy$ with $0 \leq |y| < 10\eta$ and $\eta < \frac{\pi}{40t}$, we have
\begin{align*}
    \bigg| \Log \xi\lp e^{-z} \rp - \dfrac{\pi^2}{12 z} + \dfrac{\log(2)}{2} - \dfrac{z}{24} \bigg| < 471 |z|^8.
\end{align*}
\end{lemma}

\begin{proof}
By Lemma \ref{Xi and L z-Expansions}, we have
\begin{align*}
	\Log \xi\lp e^{-z} \rp = z \sum_{m \geq 0} \left[ B\lp \lp m + \frac 12 \rp 2z \rp - B\lp \lp m+1 \rp 2z \rp \right],
\end{align*}
where $B(z) = \frac{e^{-z}}{z\lp 1 - e^{-z} \rp}$. We apply Proposition \ref{Euler-Maclaurin General Effective} with $N = 7$ and $A=1$. Noting that $M_8 = \frac{1}{30}$, $c_{-2} = 1$, and $c_{-1} = - \frac 12$, we have
\begin{align*}
	\bigg| \sum_{m \geq 0} B\lp (m+a)z \rp - \dfrac{\zeta(2,a)}{z^2} - \dfrac{I_{B,1}^*}{z} &- \dfrac{1}{2z} \lp \Log\lp z \rp + \gamma + \psi\lp a \rp \rp - \sum_{n = 0}^\infty c_n^* \dfrac{B_{n+1}(a)}{n+1} z^n \bigg| \\ &\leq \dfrac{J_{g,8}(z)}{1209600} |z|^7 + \sum_{k \geq 7} |b_k| \lp 1 + \dfrac{k!}{10 (k-7)!} \rp |z|^k,
\end{align*}
where $b_k = \frac{B_{k+2}}{(k+2)!} + \frac{(-1)^{k+1}}{2(k+1)!}$ and $g(z) = \frac{e^{-z}}{z\lp 1 - e^{-z} \rp} - \frac{1}{z^2} + \frac{e^{-z}}{2z}$. Note that like $B(z)$, the power series representation of $g(z)$ has radius of convergence $2\pi$. We now reduce the bounds on the right-hand side of the above. Setting $\alpha = \frac{3\pi}{2} \frac{z}{|z|}$, we decompose $J_{g,8}(z)$ as
\begin{align*}
	J_{g,8}(z) = \int_0^\alpha \left| g^{(8)}(w) \right| |dw| + \int_\alpha^\infty \left| g^{(8)}(w) \right| |dw|,
\end{align*}
where the paths proceed radially as originally defined. We first bound $g^{(6)}(w)$ on the interval near zero. Invoking \eqref{Bernoulli Inequality}, we can see that
\begin{align*}
	\left| b_k \right| \leq \frac{1}{12} \cdot \lp \frac{1}{2\pi} \rp^{k} + \frac{1}{2(k+1)!}.
\end{align*}
for all $k$, so for $|w| < \frac{3\pi}{2}$ we have
\begin{align*}
	\left| g^{(8)}(w) \right| \leq \sum_{k \geq 0} \dfrac{(k+8)!}{k!} \lp \frac{1}{12} \cdot \lp \frac{1}{2\pi} \rp^{k+8} + \frac{1}{2(k+9)!} \rp \lp \dfrac{3\pi}{2} \rp^k < 367.
\end{align*}
Thus,
\begin{align*}
	\int_0^\alpha \left| g^{(8)}(w) \right| |dw| < 367\dfrac{3\pi}{2} < 1730.
\end{align*}
Now, $g^{(8)}(w)$ may be written in the form
\begin{align*}
	g^{(8)}(w) = \sum_{j=1}^9 \dfrac{p_j(w)}{\lp e^w - 1 \rp^{9-j} w^j}
\end{align*}
for certain polynomials $p_j(w)$ of degree $j-1$ with non-negative coefficients. For $w$ on the major arc, we have $u = \mathrm{Re}(w) \leq |w| \leq \sqrt{101}u$, and therefore by the triangle inequality we have
\begin{align*}
	\left| g^{(8)}(w) \right| \leq \sum_{j=1}^9 \dfrac{p_j\lp \sqrt{101} u \rp}{\lp e^u - 1 \rp^{9-j} u^j}.
\end{align*}
Integrating with the aid of a computer, we have
\begin{align*}
	\int_\alpha^\infty \left| g^{(8)}(w) \right| |dw| \leq \sqrt{101} \int_{\frac{3\pi}{2}}^\infty  \sum_{j=1}^9 \dfrac{p_j\lp \sqrt{101} u \rp}{\lp e^u - 1 \rp^{9-j} u^j} du < 2206410.
\end{align*}
Therefore,
\begin{align*}
	J_{g,8}(z) < 1730 + 2206410 = 2208140.
\end{align*}
By the previous bound on $\left| b_k \right|$ as well as the fact that $|z| < \frac{2}{5}$ on the major arc, we have that
\begin{align*}
	\sum_{k \geq 7} \left|b_k\right| \lp 1 + \dfrac{k!}{10 (k-7)!} \rp |z|^{k-7} < \sum_{k \geq 7} \lp \frac{1}{12} \cdot \lp \frac{1}{2\pi} \rp^k + \frac{1}{2(k+1)!} \rp \lp 1 + \dfrac{k!}{10 (k-7)!} \rp \lp \dfrac{2}{5} \rp^{k-7} < \dfrac{1}{100}.
\end{align*}
Therefore, by letting $z \mapsto 2z$ and applying the bounds just derived, we obtain
\begin{align*}
	\bigg| \sum_{m \geq 0} B\lp (m+a)2z \rp - \dfrac{\zeta(2,a)}{4z^2} - \dfrac{I_{B,1}^*}{2z} - \dfrac{1}{4z} \lp \Log\lp 2z \rp + \gamma + \psi\lp a \rp \rp &- \sum_{n = 0}^\infty c_n^* \dfrac{B_{n+1}(a)}{n+1} 2^n z^n \bigg| < 235 |z|^7.
\end{align*}
By the expansion from Lemma \ref{Xi and L z-Expansions}, we may conclude immediately that
\begin{align*}
	\bigg| \Log \xi\lp e^{-z} \rp + \dfrac{\zeta(2,1) - \zeta\lp 2, \frac 12 \rp}{4 z} &+ \dfrac{ \psi(1) - \psi\lp \frac 12 \rp}{4} - \sum_{n = 0}^\infty c_n^* \dfrac{B_{n+1}(1) - B_{n+1}\lp \frac 12 \rp}{n+1} 2^n z^{n+1} \bigg| < 470 |z|^8.
\end{align*}

We now proceed to simplify terms in the bounds above. By the definition of $c_n^*$ along with $c_{-1} = -\frac 12$, we may calculate
\begin{align*}
	\sum_{n = 0}^\infty c_n^* \dfrac{B_{n+1}(1) - B_{n+1}\lp \frac 12 \rp}{n+1} 2^n z^{n+1} = \dfrac{z}{24} - \sum_{n \geq 7} \dfrac{(-1)^{n+1}\lp B_{n+1}(1) - B_{n+1}\lp \frac 12 \rp \rp}{(n+1) (n+1)!} 2^{n-1} z^{n+1}.
\end{align*}
Now, because of the identity $\zeta\lp s, \frac 12 \rp = \lp 2^s - 1 \rp \zeta(2)$, we have $\zeta\lp 2, 1 \rp - \zeta\lp 2, \frac 12 \rp = - \frac{\pi^2}{3}$. Furthermore, by \cite[(5.4)]{DLMF} we have 
$\psi(1) = -\gamma$ and $- \psi\lp \frac 12 \rp = -2 \log(2) - \gamma$. Therefore, using the triangle inequality in the form $|x| \leq |x-y| + |y|$ and $|z| < \frac{\pi}{2}$, we may obtain
\begin{align*}
	\bigg| \Log \xi\lp e^{-z} \rp - \dfrac{\pi^2}{12 z} + \dfrac{ \log(2)}{2} - \dfrac{z}{24} \bigg| < 470 |z|^8 + \left| \sum_{n \geq 7} \dfrac{(-1)^{n+1}\lp B_{n+1}(1) - B_{n+1}\lp \frac 12 \rp \rp}{(n+1) (n+1)!} 2^{n-1} z^{n-7} \right| \cdot |z|^8.
\end{align*}
Lehmer's bound \eqref{Bernoulli Inequality} along with the straightforward inequality $\zeta(n+1) \leq \zeta(2) = \frac{\pi^2}{6}$ for $n \geq 1$ implies that
\begin{align*}
	\dfrac{\left| B_{n+1}(1) - B_{n+1}\lp \frac 12 \rp \right|}{(n+1)!} \leq \dfrac{4 \zeta(n+1)}{(2\pi)^{n+1}} \leq \dfrac{\pi}{3 \lp 2\pi \rp^n}
\end{align*}
for $n \geq 1$. Therefore using the fact that $|z| < \frac{\pi}{2}$ on the major arc with $\eta < \frac{\pi}{40t}$, we have
\begin{align*}
	\sum_{n \geq 7} \dfrac{\left| B_{n+1}(1) - B_{n+1}\lp \frac 12 \rp \right|}{(n+1) (n+1)!} 2^{n-1} |z|^{n-7} \leq \dfrac{1}{6\pi^6} \sum_{n \geq 7} \dfrac{1}{(n+1) 2^{n-7}} < 1.
\end{align*}
Putting together all evaluations, we conclude that
\begin{align*}
	\bigg| \log \xi\lp e^{-z} \rp - \dfrac{\pi^2}{12 z} + \dfrac{\log(2)}{2} - \dfrac{z}{24} \bigg| < 471 |z|^8.
\end{align*}
This completes the proof.
\end{proof}

\begin{corollary} \label{Xi Major Arc Corollary}
For any integer $t \geq 2$ and any complex number $z = \eta + iy$ satisfying $0 \leq |y| < 10\eta$ and $\eta < \frac{\pi}{40t}$, we have
\begin{align*}
    \left| \xi\lp e^{-z} \rp - \exp\lp \dfrac{\pi^2}{12z} - \dfrac{\log(2)}{2} + \dfrac{z}{24} \rp \right| < \dfrac{630 |z|^8}{\sqrt{2}} \exp\lp \dfrac{\pi^2}{12|z|} \rp.
\end{align*}
\end{corollary}

\begin{proof}
Suppose $f(z), g(z), e(z)$ are any three functions that satisfy
\begin{align*}
    \left| \Log f(z) - \Log g(z) \right| \leq e(z)
\end{align*}
for $|z| < \frac{\pi}{40t}$. Note that we may factorize
\begin{align*}
    \left| f(z) - g(z) \right| = \left| \exp\lp \Log f(z) - \Log g(z) \rp - 1 \right| \cdot \left| g(z) \right|.
\end{align*}
Applying this factorization with $f(z) := \xi\lp e^{-z} \rp$ and $g(z) := \exp\lp \frac{\pi^2}{12 z} - \frac{\log(2)}{2} + \frac{z}{24} \rp$ will give the result. Using Lemma \ref{Xi Major Arc} and Taylor series, we have
\begin{align*}
    \left| \exp\lp \Log \xi\lp e^{-z} \rp - \dfrac{\pi^2}{12 z} + \dfrac{\log(2)}{2} - \dfrac{z}{24} \rp - 1 \right| < \sum_{n \geq 1} \dfrac{1}{n!} \lp 471 |z|^8 \rp^n = \exp\lp 471 |z|^8 \rp - 1.
\end{align*}
For $|z| < \frac{\sqrt{101} \pi}{80}$, we have $471 |z|^8 < 0.28$, and since $e^x - 1 < \frac{4}{3} x$ for $0 < x < 0.55$, we have
\begin{align*}
	\left| \exp\lp \Log \xi\lp e^{-z} \rp - \dfrac{\pi^2}{12 z} + \dfrac{\log(2)}{2} - \dfrac{z}{24} \rp - 1 \right| < 628 |z|^8.
\end{align*}
Using $\eta \leq |z|$ and $\eta < \frac{\pi}{80}$, we may conclude that
\begin{align*}
    \left| \exp\lp \dfrac{\pi^2}{12 z} - \dfrac{\log(2)}{2} + \dfrac{z}{24} \rp \right| \leq \exp\lp \dfrac{\pi^2 \eta}{12|z|^2} + \dfrac{\eta}{24} \rp < \dfrac{501}{500\sqrt{2}} \exp\lp \dfrac{\pi}{12|z|}\rp.
\end{align*}
Combining the given bounds completes the proof.
\end{proof}

\subsection{Minor arc effective bounds}

We now calculate effective bounds on both $\xi(q)$ and $L_{r,t}(q)$ for the minor arc $10\eta \leq |y| < \pi$, subject to the additional constraint $\eta < \frac{\pi}{40t} \leq \frac{\pi}{80}$.

\begin{lemma} \label{Xi Minor Arc Bound}
Let $t \geq 2$ be an integer. Assume $z = \eta + iy$ satisfies $10 \eta \leq |y| < \pi$ and $0 < \eta < \frac{\pi}{40t}$. Then we have
\begin{align*}
    \left| \xi\lp e^{-z} \rp \right| < \exp\lp \dfrac{41}{50\eta} \rp.
\end{align*}
\end{lemma}

\begin{proof}
Let $q = e^{-z}$. Recall that
\begin{align*}
	\Log \xi(q) = - \sum_{m \geq 1} \dfrac{(-1)^m q^m}{m\lp 1 - q^m \rp}.
\end{align*}
By taking absolute values and splitting off the $m=1$ term and noting that $\log P\lp |q| \rp = \sum_{m \geq 1} \frac{|q|^m}{m \lp 1 - |q|^m \rp}$, we have
\begin{align*}
	\left| \Log \xi(q) \right| \leq \Log P\lp |q| \rp - |q| \lp \dfrac{1}{1 - |q|} - \dfrac{1}{|1 - q|} \rp,
\end{align*}
where $P(q) = (q;q)_\infty^{-1}$. To bound $\Log P\lp |q| \rp$, we recall that $|q| = e^{-\eta}$ and use the series expansion
\begin{align*}
	\Log P\lp |q| \rp = \sum_{m \geq 1} \dfrac{|q|^m}{m\lp 1 - |q|^m \rp} = \sum_{m \geq 1} \dfrac{e^{-mx}}{m\lp 1 - e^{-mx} \rp}.
\end{align*}
From the fact that $\frac{e^{-x}}{1 - e^{-x}} < \frac{1}{x}$ for all $x > 0$, we may therefore deduce that
\begin{align} \label{Log-P Bound}
	\Log P\lp |q| \rp < \sum_{m \geq 1} \dfrac{1}{m^2 \eta} = \dfrac{\pi^2}{6\eta}.
\end{align}
Now, we have $\left| 1 - q \right|^2 = 1 - 2 \cos(y) e^{-\eta} + e^{-2\eta}$. In the region $10\eta \leq |y| < \pi$, we have by the fact that $\cos(x)$ is decreasing for $0 < x < \pi$ that $\left| 1 - q \right|^2 \geq 1 - 2 \cos(10\eta) e^{-\eta} + e^{-2\eta}$. It can be checked in an elementary manner that $1 - 2 \cos(10\eta) e^{-\eta} + e^{-2\eta} > 95 \eta^2$, and so we have $\left| 1 - q \right| > \sqrt{95} \eta$. By using the bound $1 - |q| = 1 - e^{-\eta} \leq \eta$, we have for $10\eta \leq |y| < \pi$ and $\eta < \frac{\pi}{80}$ that
\begin{align*}
	\dfrac{1}{|1 - q|} - \dfrac{1}{1 - |q|} < \lp \dfrac{1}{\sqrt{95}} - 1 \rp \dfrac{1}{\eta}.
\end{align*}
Therefore, using $|q| \leq 1$ we have
\begin{align*}
	\left|\Log\lp \xi\lp q \rp \rp \right| \leq \lp \dfrac{\pi^2}{6} + \dfrac{1}{\sqrt{95}} - 1 \rp \dfrac{1}{\eta} < \dfrac{3}{4\eta}.
\end{align*}
Exponentiating completes the proof.
\end{proof}

\begin{lemma} \label{L Minor Arc Bound}
Let $t \geq 2$ and $0 < r \leq t$ be integers. Assume $z = \eta + iy$ is a complex number satisfying $\eta > 0$. Then we have
\begin{align*}
    \left| L_{r,t}\lp e^{-z} \rp \right| < \dfrac{1}{\eta^2}.
\end{align*}
\end{lemma}

\begin{proof}
Let $q = e^{-z}$ and let $\sigma_0(n) = \sum_{d|n} 1$ be the standard divisor counting function. Then 
\begin{align*}
	\left| L_{r,t}(q) \right| \leq \sum_{m \geq 1} \dfrac{e^{-m\eta}}{1 - e^{-m\eta}} = \sum_{m \geq 1} \sigma_0(m) e^{-m\eta} \leq \sum_{m \geq 1} m e^{-m\eta} = \dfrac{e^\eta}{\lp e^\eta - 1 \rp^2} < \dfrac{1}{\eta^2}.
\end{align*}
This completes the proof.
\end{proof}

\section{Proofs of the Main Results} \label{Effective Proofs}

In this section, we prove the main results. We first prove Theorem \ref{MAIN} and Theorem \ref{Effective MAIN}. We then show how Theorem \ref{Effective MAIN} can be used to prove Corollary \ref{Effective Inequality}.

\subsection{Proof of Theorem \ref{MAIN}}

By Lemmas \ref{L Major Arc} and \ref{Xi Major Arc}, we have $L_{r,t}\lp e^{-z} \rp = \frac{\log(2)}{tz} - \frac{1}{2}\lp \frac{r}{t} - \frac{1}{2} \rp + O(z)$ and $\xi\lp e^{-z} \rp = \frac{1}{\sqrt{2}} e^{\frac{\pi^2}{12z} + O(z)}$ on the major arc. These imply asymptotics for $\mathcal D_{r,t}(q) = L_{r,t}(q) \xi(q)$, which is the generating function for $D_{r,t}(n)$ by Lemma \ref{Generating Function}, that satisfies (1) in Proposition \ref{Wright Circle Method}. Lemmas \ref{L Minor Arc Bound} and \ref{Xi Minor Arc Bound} imply condition (2), and so we may apply Proposition \ref{Wright Circle Method}, which yields the claimed asymptotic formula.

\subsection{Proof of Theorem \ref{Effective MAIN}}

In this section, we complete the proof of Theorem \ref{Effective MAIN} by following the proof of \cite[Proposition 1.8]{NgoRhoades} (which is a version of Wright's circle method slightly different from Proposition \ref{Wright Circle Method}) and making the bounds in each step effective. Let $\mathcal{C}$ be the circle in the complex plane with center 0 and radius $e^{-\eta}$, where $\eta = \frac{\pi}{\sqrt{12n}}$. By Cauchy's formula and Lemma \ref{Generating Function}, we have
\begin{align*}
    D_{r,t}(n) = \dfrac{1}{2\pi i} \int_{\mathcal C} \dfrac{\mathcal{D}_{r,t}(q)}{q^{n+1}} dq = \dfrac{1}{2\pi i} \int_{\mathcal C} \dfrac{L_{r,t}(q) \xi(q)}{q^{n+1}} dq.
\end{align*}
Throughout, we fix $q = e^{-z}$ with $z = \eta + iy$, so that $|q| = e^{-\eta}$. We will estimate $D_{r,t}(n)$ by decomposing this integral into convenient pieces. Choose $\delta > 0$ so that for the major arc $\mathcal{C}_1$, $z = \eta + iy \in \mathcal{C}_1$ satisfies $0 < |y| < 10\eta$. We shall also assume that $\eta < \frac{\pi}{40t}$, which is equivalent to the bound $n > \frac{400t^2}{3}$.

Let $\mathcal{C}_2 := \mathcal{C} \backslash \mathcal{C}_1$ denote the minor arc. Define for $s \geq 0$ the integrals
\begin{align} \label{V Definition}
    V_s(n) &:= \dfrac{1}{2\pi i} \int_{\mathcal{C}_1} \dfrac{z^{s-1}}{q^{n+1}} \exp\lp \dfrac{\pi^2}{12z} - \dfrac{\log(2)}{2} + \dfrac{z}{24} \rp dq \notag \\ &= \dfrac{1}{2\pi \sqrt{2} i} \int_{D_0} z^{s-1} \exp\lp \dfrac{\pi^2}{12z} + \lp n + \dfrac{1}{24} \rp z \rp dz.
\end{align}
We use the integrals $V_s(n)$ to estimate $D_{r,t}(n)$. In particular, we have the decomposition
\begin{align*}
    D_{r,t}(n) - \alpha_0 V_0(n) - \alpha_{1,r} V_1(n) - \alpha_{2,r} V_2(n) - \alpha_{4,r} V_4(n) = E_1 + E_2 + E_3,
\end{align*}
where $\alpha_0 = \frac{\log(2)}{t}$, $\alpha_{1,r} = - \frac{1}{2} B_1\lp \frac rt \rp$, $\alpha_{2,r} = \frac{t}{8} B_2\lp \frac rt \rp$, $\alpha_{4,r} = \frac{t^3}{192} B_4\lp \frac rt \rp$, and
\begin{align*}
    E_1 &:= \dfrac{1}{2\pi i} \int_{\mathcal{C}_2} \dfrac{L_{r,t}(q) \xi(q)}{q^{n+1}} dq, \\
    E_2 &:= \dfrac{1}{2\pi i} \int_{\mathcal{C}_1} \dfrac{L_{r,t}(q) \lp \xi(q) - \exp \lp \dfrac{\pi^2}{12z} - \dfrac{\log(2)}{2} + \dfrac{z}{24} \rp \rp}{q^{n+1}} dq, \\
    E_3 &:= \dfrac{1}{2\pi i} \int_{\mathcal{C}_1} \dfrac{\lp L_{r,t}(q) - \alpha_0 z^{-1} - \alpha_1 - \alpha_2 z  - \alpha_4 z^3 \rp \exp \lp \dfrac{\pi^2}{12z} - \dfrac{\log(2)}{2} + \dfrac{z}{24} \rp}{q^{n+1}} dq.
\end{align*}
Although $\alpha_{i,r}$ for $i > 0$ depends on $r$, we suppress this dependence when $r$ is understood from context. Because $|z|^2 \geq \eta^2$ and $\eta = \frac{\pi}{\sqrt{12n}}$, we have
\begin{align*}
    \left| \exp\lp \dfrac{\pi^2}{12z} + nz \rp \right| = \exp\lp \dfrac{\pi^2 \eta}{12|z|^2} + n \eta \rp \leq \exp\lp \pi \sqrt{\dfrac{n}{3}} \rp.
\end{align*}
Furthermore, we note that
\begin{align*}
    \left| \int_{\mathcal{C}_1} q^{-1} dq \right| = \left| \left[ \Log\lp e^{-z} \rp \right]_{z = \eta - 10 \eta i}^{z = \eta + 10 \eta i} \right| \leq 20 \eta.
\end{align*}
and
\begin{align*}
	\left| \int_{\mathcal{C}_2} q^{-1} dq \right| \leq \mathrm{len}(\mathcal{C}_2) \cdot \max\limits_{z \in \mathcal{C}_2} |z| \leq 4.2 \pi^2.
\end{align*}
We also note that on the major arc $0 < |y| < 10\eta < \pi$ we have $\eta \leq |z| < \sqrt{101} \eta$. Since $\eta < \frac{\pi}{80}$, we also have $|z| < \frac{\sqrt{101}\pi}{80} < \frac{2}{5}$. These inequalities will be used freely in what follows.

To bound $E_3$, we recall that Lemma \ref{L Major Arc} says that for $\eta < \frac{\pi}{40t}$ on the major arc, we have
\begin{align*}
    \left| L_{r,t}(q) - \alpha_0 z^{-1} - \alpha_1 - \alpha_2 z - \alpha_4 z^3 \right| < \dfrac{1}{20} t^5 |z|^5,
\end{align*}
and we therefore have using these equations and numerical estimates that
\begin{align*}
    \left| E_3 \right| &\leq \dfrac{10 \eta}{\pi} \left| L_{r,t}(q) - \alpha_0 z^{-1} - \alpha_1 - \alpha_2 z - \alpha_4 z^3 \right| \left| \exp\lp \dfrac{\pi^2}{12z} - \dfrac{\log(2)}{2} + nz + \dfrac{z}{24} \rp \right| \\ &< \dfrac{14381 t^5}{n^3} \exp\lp \pi \sqrt{\dfrac{n}{3}} \rp.
\end{align*}
To bound $E_2$, we apply Corollary \ref{Xi Major Arc Corollary}, which we recall says
\begin{align*}
	\left| \xi\lp e^{-z} \rp - \exp\lp \dfrac{\pi^2}{12z} - \dfrac{\log(2)}{2} + \dfrac{z}{24} \rp \right| < \dfrac{630|z|^8}{\sqrt{2}} \exp\lp \dfrac{\pi^2}{12|z|} \rp.
\end{align*}
Therefore,
\begin{align*}
    \left| E_2 \right| \leq \dfrac{10 \eta}{\pi}  \left| L_{r,t}(q) \right| \left| \xi(q) - \exp\lp \dfrac{\pi^2}{12z} - \dfrac{\log(2)}{2} + \dfrac{z}{24} \rp \right| \left| \exp\lp nz \rp \right| < \dfrac{945285959087}{t n^4} \exp\lp \pi \sqrt{\dfrac{n}{3}} \rp.
\end{align*}
Finally, using Lemmas \ref{Xi Minor Arc Bound} and \ref{L Minor Arc Bound} we have
\begin{align*}
    \left| E_1 \right| \leq \dfrac{4.2 \pi^2}{2\pi} \left| L_{r,t}(q) \right| \left| \xi(q) \right| \left| \exp\lp nz \rp \right| < 9 n \exp\lp \lp \dfrac{3 \sqrt{3}}{2\pi} + \dfrac{\pi}{\sqrt{12}} \rp \sqrt{n} \rp.
\end{align*}
We have therefore shown that
\begin{align*}
    \left| D_{r,t}(n) - \alpha_0 V_0(n) - \alpha_1 V_1(n) - \alpha_2 V_2(n) - \alpha_4 V_4(n) \right| \leq \mathrm{Err}_t(n)
\end{align*}
where
\begin{align} \label{E_t Definition}
    \mathrm{Err}_t(n) := \dfrac{14381 t^5}{n^3} \exp\lp \pi \sqrt{\dfrac{n}{3}} \rp + \dfrac{945285959087}{t n^4} \exp\lp \pi \sqrt{\dfrac{n}{3}} \rp + 9 n \exp\lp \lp \dfrac{3 \sqrt{3}}{2\pi} + \dfrac{\pi}{\sqrt{12}} \rp \sqrt{n} \rp.
\end{align}
This completes the proof of Theorem \ref{Effective MAIN}. \\

\subsection{Proof of Corollary \ref{Effective Inequality}}
We now wish to resolve the inequality $D_{r,t}(n) \geq D_{s,t}(n)$ for integers $n \geq 0$ and $0 < r < s \leq t$. We define for convenience $\alpha_{j,r}^* := \alpha_{j,r} - \alpha_{j,r+1}$ and
\begin{align*}
    M_{r,t}(n) := \alpha_0 V_0(n) + \alpha_{1,r} V_1(n) + \alpha_{2,r} V_2(n) + \alpha_{4,r} V_4(n).
\end{align*}
Note that since $D_{r,t}(n) - D_{s,t}(n) = \sum_{j=r}^{s-1} D_{j,t}(n) - D_{j+1,t}(n)$, it suffices to prove $D_{r,t}(n) \geq D_{r+1,t}(n)$ for all $n > 8$ and $0 < r < t$. We therefore focus on this inequality.

By Theorem \ref{Effective MAIN} applied to both terms in $D_{r,t}(n) - D_{r+1,t}(n)$, in order to show $D_{r,t}(n) - D_{r+1,t}(n) \geq 0$ it suffices to show
\begin{align*}
    M_{r,t}(n) - M_{r+1,t}(n) \geq 2 \mathrm{Err}_t(n).
\end{align*}
Collecting together like terms and simplifying, this is equivalent to
\begin{align*} 
    \alpha_{1,r}^* V_1(n) + \alpha_{2,r}^* V_2(n) + \alpha_{4,r}^* V_4(n) \geq 2 \mathrm{Err}_t(n).
\end{align*}
We also wish to bound the terms $\alpha_{j,r}^*$ for $j = 1,2,4$. Since $B_1(x) = x - \frac 12$, $B_2(x) = x^2 - x + \frac 16$, and $B_4(x) = x^4 - 2x^3 + x^2 - \frac{1}{30}$, and $1 \leq r < t$ (since $r+1 \leq t$), we have $\alpha_{1,r}^* = \frac{1}{2t}$, $\alpha_{2,r}^* = \frac{t-2r-1}{8t^2} \geq - \frac{3}{16}$, and $\alpha_{4,r}^* \geq - \frac{233}{48}$ for $2 \leq t \leq 10$. Therefore, it would suffice to prove that
\begin{align*}
    \dfrac{V_1(n)}{2t} \geq \dfrac{3}{16} V_2(n) + \dfrac{233}{48} V_4(n) + 2 \mathrm{Err}_t(n).
\end{align*}

Now, note that by the definition of $\widehat{I}_s(n)$ used in Lemma \ref{Bessel Estimates}, we have
\begin{align*}
	V_s(n) = \dfrac{1}{\sqrt{2}} \lp \dfrac{24n+1}{2\pi^2} \rp^{- \frac s2} \widehat{I}_{-s}(n),
\end{align*}
and therefore by Lemma \ref{Bessel Estimates} we may conclude that for $s \geq 1$,
\begin{align*}
	\left| V_s(n) - \dfrac{1}{\sqrt{2}} \lp \dfrac{24n+1}{2\pi^2} \rp^{- \frac s2} I_{-s}\lp \pi \sqrt{\dfrac{1}{3}\lp n + \dfrac{1}{24} \rp} \rp \right| \leq \sqrt{2} \exp\lp \dfrac{3\pi}{4} \sqrt{\dfrac{n}{3}} \rp \int_0^\infty \lp 10 + u \rp^{s - 1} e^{-\lp n + \frac{1}{24} \rp u} du.
\end{align*}
Now by a substitution $u \mapsto \frac{1}{n + \frac{1}{24}}u$, we have
\begin{align*}
	\sqrt{2} \int_0^\infty \lp 10 + u \rp^{-\nu - 1} e^{-\lp n + \frac{1}{24} \rp u} du = \dfrac{24\sqrt{2}}{24n+1} \int_0^\infty \lp 10 + \dfrac{24u}{24n+1} \rp^{s - 1} e^{-u} du
\end{align*}
For $\beta_1 = 1$, $\beta_2 = 11$, and $\beta_4 = 1349$, we may conclude that each of $s = 1, 2, 4$ satisfies
\begin{align*}
	\left| V_s(n) - \dfrac{1}{\sqrt{2}} \lp \dfrac{24n+1}{2\pi^2} \rp^{- \frac s2} I_{-s}\lp \pi \sqrt{\dfrac{1}{3}\lp n + \dfrac{1}{24} \rp} \rp \right| < \dfrac{24 \beta_s \sqrt{2}}{24n+1} \exp\lp \dfrac{3\pi}{4} \sqrt{\dfrac{n}{3}} \rp.
\end{align*}
Therefore, if we set $n' := n + \frac{1}{24}$ for convenience, to prove the desired inequality it would suffice to show that
\begin{align} \label{Reduced Inequality}
	\dfrac{\pi}{4t\sqrt{6n'}} I_{-1}\lp \pi \sqrt{\dfrac{n'}{3}} \rp &\geq \dfrac{\pi^2}{64 n' \sqrt{2}} I_{-2}\lp \pi \sqrt{\dfrac{n'}{3}} \rp + \dfrac{233\pi^4}{6912 \sqrt{2} (n')^2} I_{-4}\lp \pi \sqrt{\dfrac{n'}{3}} \rp \notag \\ &+ \lp \dfrac{1}{t \sqrt{2}} + \dfrac{33 \sqrt{2}}{16} + \dfrac{314317 \sqrt{2}}{48} \rp \dfrac{1}{n'} \exp\lp \dfrac{3\pi}{4} \sqrt{\dfrac{n}{3}} \rp + 2 \mathrm{Err}_t(n).
\end{align}

In summary, we have shown that in order to show that $D_{r,t}(n) \geq D_{s,t}(n)$ for all $0 < r < s \leq t$ for a fixed value of $n$, it suffices to consider the case $s = r+1$ for each $r$, and all of these cases follow from the inequality \eqref{Reduced Inequality} is true. In the process of deriving \eqref{Reduced Inequality}, we have also assumed $n > \frac{400t^2}{3}$. Therefore, we define the integer $N_t(n)$ as the smallest positive integer satisfying $N_t(n) > \frac{400t^2}{3}$ and so that \eqref{Reduced Inequality} is true for all $n > N_t(n)$, from which it follows that $D_{r,t}(n) \geq D_{s,t}(n)$ for all $n > N_t(n)$. The table below gives values of $N_t(n)$, which are computed with the aid of a computer. \\

\begin{center}
\begin{tabular}{|c|c|c|c|c|c|c|c|c|c|c|} \hline
    $t$ & 2 & 3 & 4 & 5 & 6 & 7 & 8 & 9 & 10 \\ \hline
    $N_t(n)$ & 108077 & 112183 & 115240 & 117804 & 120247 & 122994 & 126772 & 133268 & 147752 \\ \hline
\end{tabular}

\vspace{0.1in}

{\small Table 2: Numerics for Corollary \ref{Effective Inequality}.}

\vspace{0.1in}
\end{center}

It therefore only remains to check the possible values of $D_{r,t}(n) - D_{r+1,t}(n)$ for $n \leq N_t(n)$ by computer and determine all possible counterexamples which arise from these cases. All such counterexamples satisfy $n \leq 8$ for $2 \leq t \leq 10$, which completes the proof.

\end{document}